\documentclass[twoside,leqno,amsmath,jamsl]{amsart}
\usepackage{amsfonts,amsthm,amssymb}
\newtheorem{theorem}{Theorem}

\newtheorem*{corollary}{Corollary}
\newtheorem{proposition}{Proposition}
\begin{document}

\title[Kimmerle conjecture for Held and O'Nan simple groups]{Kimmerle conjecture for the Held\\ and O'Nan sporadic simple groups}
\author{V.~Bovdi, A.~Grishkov, A.~Konovalov}
\date{November 3rd, 2008}
\dedicatory{Dedicated to 65th birthday of Professor Pali D\"om\"osi}

\address{
V.A.~Bovdi\newline	
Institute of Mathematics, University of
Debrecen\newline P.O.  Box 12, H-4010 Debrecen,
Hungary\newline Institute of Mathematics and
Informatics, College of Ny\'\i regyh\'aza\newline
S\'ost\'oi \'ut 31/b, H-4410 Ny\'\i regyh\'aza, Hungary}
\email{vbovdi@math.klte.hu}

\address{
A.~Grishkov
\newline Departamento de Matem\'atica, (IME-USP),
\newline  Rua do Matao, 1010 - Cidade Universit\'aria,
\newline  CEP 05508-090, Sao Paulo - SP,  Brasil}
\email{grishkov@ime.usb.br}

\address{
A.B.~Konovalov
\newline School of Computer Science, University of St Andrews,
\newline Jack Cole Building, North Haugh, St Andrews, Fife, KY16 9SX, Scotland}
\email{alexk@mcs.st-andrews.ac.uk}


\thanks{
The research was supported by OTKA No.K68383, FAPESP Brasil (proc.08/54650-8),
RFFI 07-01-00392A and the Royal Society of Edinburgh International Exchange Programme}


\maketitle

\begin{abstract}
Using the Luthar--Passi method, we investigate the Zassenhaus and
Kimmerle conjectures for normalized unit groups of integral group rings of the
Held and O'Nan sporadic simple groups. We confirm the Kimmerle conjecture for 
the Held simple group and also derive for both groups some extra information 
relevant to the classical Zassenhaus conjecture.
\end{abstract}

Let $U(\mathbb Z G)$ be the unit group of the integral
group ring $\mathbb Z G$ of a finite group $G$. It is well known
that $U(\mathbb Z G) = U(\mathbb Z) \times V(\mathbb Z G)$, where
$$
V(\mathbb ZG) = \Big\{ \; \sum_{g \in G}\alpha_g g \in U(\mathbb ZG)
\; \mid \; \sum_{g \in G} \alpha_g=1, \; \alpha_g \in \mathbb Z \; \Big\}.
$$
is the normalized unit group of $U(\mathbb Z G)$.
Throughout the paper (unless stated otherwise) the unit is
always normalized and not equal to the identity element of $G$.

One of most
interesting conjectures in the theory of integral group ring is
the conjecture {\bf (ZC)} of  H.~Zassenhaus \cite{Zassenhaus},
saying that every normalized torsion unit $u\in V(\mathbb ZG)$ is conjugate
to an element in $G$ within the rational group algebra $\mathbb Q
G$.

For finite simple groups, the main tool of  the investigation of
the Zassenhaus conjecture {\bf (ZC)} is the Luthar--Passi method, introduced
in \cite{Luthar-Passi} to solve this conjecture for the alternating group 
$A_{5}$. Later in \cite{Hertweck1} M.~Hertweck applied it for some other 
groups using $p$-Brauer characters, and then extended the previous result
by M.~Salim \cite{Salim} to confirm {\bf (ZC)} for the 
alternating group $A_{6}$ in \cite{Hertweck-A6} 
(note that for larger alternating groups the problem is still open).
The method also proved to be useful for groups containing non-trivial normal
subgroups as well (see related results in \cite{Artamonov-Bovdi, Hertweck2,
Hertweck3, Hertweck1, Kimmerle, Luthar-Trama}).

One of the variations of {\bf (ZC)} was formulated by W.~Kimmerle in \cite{Kimmerle}.
Denote by $\#(G)$ the set of all primes dividing the order of $G$.
The Gruenberg--Kegel graph (or the prime graph) of $G$ is the graph
$\pi (G)$ with vertices labelled by the primes in $\# (G)$ and
there is an edge from $p$ to $q$ if and only if there is an element of order
$pq$ in the group $G$. Then W.~Kimmerle asked the following:

\vspace{3pt}
\centerline {{\bf Conjecture (KC)}: Is it true that $\pi (G) =\pi (V(\mathbb Z G))$
for any finite group $G$?}
\vspace{3pt}

It is easy to see that the Zassenhaus conjecture {\bf (ZC)} implies the
Kimmerle conjecture {\bf (KC)}. In \cite{Kimmerle} W.~Kimmerle confirmed
{\bf (KC)} for finite Frobenius and solvable groups.
Recently {\bf (KC)} was confirmed for some simple groups (see \cite{Hertweck-A6,Hertweck1}), including 12 of 26 sporadic simple groups (see \cite{Bovdi-Jespers-Konovalov-J1-J2-J3,
Bovdi-Konovalov-M11, Bovdi-Konovalov-McL,Bovdi-Konovalov-M23, Bovdi-Konovalov-Ru,
Bovdi-Konovalov-Linton-M22, Bovdi-Konovalov-Marcos-Suz, Bovdi-Konovalov-Siciliano-M12}).

In the present paper we confirm {\bf(KC)} for the Held
sporadic simple group $\verb"He"$ \cite{Held,Stroth}, using the Luthar--Passi method
as the main tool. We also study the same problem for the
O'Nan sporadic simple group $\verb"ON"$ \cite{ON}, and prove the non-existence
of torsion units of all orders relevant to {{\bf(KC)}} except
orders 33 and 57.
Additionally, we derive certain information about possible torsion
units in $V(\mathbb Z [\verb"He"])$ and $V(\mathbb Z [\verb"ON"])$
and their partial augmentations, which will be useful for
further investigation of {\bf(ZC)} for these groups.
The development version of the GAP package LAGUNA
\cite{LAGUNA} was very helpful to speed up computational work and
derive purely theoretical arguments for the proof.

First we introduce some notation.
Let $G$ be a group. Let $\mathcal{C} =\{ C_{1}, \ldots, C_{nt},
\ldots \}$ be  the collection of all conjugacy classes of $G$,
where the first index denotes the order of the elements of this
conjugacy class and $C_{1}=\{ 1\}$. Suppose $u=\sum \alpha_g g \in
V(\mathbb Z G)$ be a non-trivial unit of finite order $k$. Denote by
$$\nu_{nt}=\nu_{nt}(u)=\varepsilon_{C_{nt}}(u)=\sum_{g\in C_{nt}}
\alpha_{g},$$ the partial augmentation of $u$ with respect to the
conjugacy class $C_{nt}$. From the Berman--Higman Theorem
(see \cite{Artamonov-Bovdi})
one knows that
$\nu_1 =\alpha_{1}=0$, so the sum of remaining partial augmentations
for non-trivial conjugacy classes is equal to one:
\begin{equation}\label{E:1}
\sum_{C_{nt}\in \mathcal{C} \atop C_{nt} \not= C_1 } \nu_{nt}=1.
\end{equation}
Clearly, for any character
$\chi$ of $G$, we get that $\chi(u)=\sum
\nu_{nt}\chi(h_{nt})$, where $h_{nt}$ is a
representative of a conjugacy class $ C_{nt}$.

The main results are the following.

%
%

\begin{theorem}\label{T:1}
Let $G$ be the Held sporadic simple group $\verb"He"$.
Let $\mathfrak{P}(u)$ be the tuple of partial augmentations of
a torsion unit $u \in V(\mathbb ZG)$ of order $|u|$,
corresponding to all non-central conjugacy classes of 
the group $G$, that is
\[
\begin{split}
\mathfrak{P}(u) =
( \nu_{2a} , \; \nu_{2b} , & \; \nu_{3a} , \; \nu_{3b} , \; \nu_{4a} , \; \nu_{4b} , \;
  \nu_{4c} , \; \nu_{5a} , \; \nu_{6a} , \; \nu_{6b} , \; \nu_{7a} , \; \nu_{7b} , \;
  \nu_{7c} , \\ \nu_{7d} , & \; \nu_{7e} , \; \nu_{8a} , \; \nu_{10a} , \; \nu_{12a} , \;
  \nu_{12b} , \; \nu_{14a} , \; \nu_{14b} , \; \nu_{14c} , \; \nu_{14d} , \\ & \; \nu_{15a} ,
  \; \nu_{17a} , \; \nu_{17b} , \;
  \; \nu_{21a} , \; \nu_{21b} , \; \nu_{21c} , \; \nu_{21d} ,
  \; \nu_{28a} , \; \nu_{28b} ) \in \mathbb Z^{32}.
\end{split}
\]
The following properties hold.

\begin{itemize}

\item[(i)]
There is no elements of orders $34$, $35$, $51$, $85$ and
$119$ in $V(\mathbb ZG)$.  Equivalently,
if  $|u|\not\in\{20, 24, 30, 40, 42, 56, 60, 84, 120, 168\}$, then $|u|$ coincides
with the order of some  $g\in G$.

\item[(ii)] If $|u|=2$, the tuple of the  partial augmentations
of $u$ belongs to the set
\[
\begin{split}
\big\{\; \mathfrak{P}(u) \; \mid \;
-6 \leq \nu_{2a} \leq 6, \; \nu_{2a} + \nu_{2b} = 1, \quad
\nu_{kx}=0, \; kx\not\in\{2a,2b\} & \; \big\}.
\end{split}
\]

\item[(iii)]If $|u| = 5$, then $u$ is rationally conjugate
to some $g\in G$.

\item[(iv)]
 If $|u|=3$, the tuple of the partial augmentations
of $u$ belongs to the set
\[
\begin{split}
\big\{\; \mathfrak{P}(u) \; \mid \;
-4 \leq \nu_{3a} \leq 5, \; \nu_{3a} + \nu_{3b} = 1, \quad
\nu_{kx}=0, \; kx\not\in\{3a,3b\} & \; \big\}.
\end{split}
\]

\item[(v)] If $|u|=17$, the tuple of the  partial augmentations of
$u$ belongs to the set
\[
\begin{split}
\big\{\; \mathfrak{P}(u) \; \mid \;
-14 \leq \nu_{17a} \leq 15, \; \nu_{17a} + \nu_{17b} = 1, \;
\nu_{kx}=0, \; kx\not\in\{17a,17b\} & \; \big\}.
\end{split}
\]

\end{itemize}
\end{theorem}

%
%

\begin{corollary} If $G$ is the Held sporadic simple group, then
$\pi(G)=\pi(V(\mathbb ZG))$.
\end{corollary}

%
%

\begin{theorem}\label{T:2}
Let $G$ be the O'Nan sporadic simple group $\verb"ON"$.
Let $\mathfrak{P}(u)$ be the tuple of partial augmentations of
a torsion unit $u \in V(\mathbb ZG)$ of order $|u|$, 
corresponding to all non-central conjugacy classes of 
the group $G$, that is
\[
\begin{split}
\mathfrak{P}(u) =
( \nu_{2a} , \; & \nu_{3a} , \; \nu_{4a} , \; \nu_{4b} , \; \nu_{5a} , \; \nu_{6a} , \;
\nu_{7a} , \; \nu_{7b} , \; \nu_{8a} , \; \nu_{8b} , \; \nu_{10a}, \\
\; & \nu_{11a} , \; \nu_{12a}, \;
\nu_{14a} , \; \nu_{15a} , \; \nu_{15b} , \; \nu_{16a} , \; \nu_{16b} , \; \nu_{16c} , \; \nu_{16d} , \\
\; & \qquad
\nu_{19a} , \; \nu_{19b} , \; \nu_{19c} , \; \nu_{20a} , \; \nu_{20b} , \; \nu_{28a} , \; \nu_{28b} , \;
\nu_{31a} , \; \nu_{31b} ) \in \mathbb Z^{29}.
\end{split}
\]
The following properties hold.

\begin{itemize}

\item[(i)]
There is no elements of orders $21$, $22$, $35$, $38$, $55$, $62$, $77$, $93$,
$95$, $133$, $155$, $209$, $217$, $341$ and $589$ in $V(\mathbb ZG)$.
Equivalently, if $$|u|\not\in\{24, 30, 33, 40, 48, 56, 57, 60, 80, 112, 120, 240\},$$
then $|u|$ coincides
with the order of some $g\in G$.

\item[(ii)]If $|u| \in \{ 2,3,5,11 \}$, then $u$ is rationally conjugate
to some $g\in G$.

\item[(iii)] If $|u|=7$, the tuple of the  partial augmentations
of $u$ belongs to the set
\[
\begin{split}
\big\{\; \mathfrak{P}(u) \; \mid \;
-3 \leq \nu_{7a} \leq 22, \; \nu_{7a} + \nu_{7b} = 1, \quad
\nu_{kx}=0, \; kx\not\in\{7a,7b\} & \; \big\}.
\end{split}
\]

\item[(iv)] If $|u|=31$, the tuple of the  partial augmentations
of $u$ belongs to the set
\[
\begin{split}
\big\{\; \mathfrak{P}(u) \; \mid \;
-39 \leq \nu_{31a} \leq 40, \; \nu_{31a} + \nu_{31b} = 1, \quad
\nu_{kx}=0, \; kx\not\in\{31a,31b\} & \; \big\}.
\end{split}
\]

\item[(v)] If $|u|=33$, the tuple of the partial augmentations
of $u$ belongs to the set
\[
\begin{split}
\big\{\; \mathfrak{P}(u) \; \mid \;
(\nu_{3a},\nu_{11a}) = (12,-11), \quad
\nu_{kx}=0, \; kx\not\in\{3a,11a\} & \; \big\}.
\end{split}
\]

\item[(vi)] If $|u|=57$, then tuple of the partial augmentations
of $u$ belongs to the set
\[
\begin{split}
\big\{\; \mathfrak{P}(u) \; \mid \;
\nu_{3a}=-18, \quad & \nu_{19a}+\nu_{19b}+\nu_{19c} = 19, \quad \\
& \nu_{kx}=0, \quad kx\not\in\{3a,19a,19b,19c\} \; \big\}.
\end{split}
\]

\end{itemize}
\end{theorem}

For the proof we will need the following results.
The first one relates the solution of
the Zassenhaus conjecture to vanishing of
partial augmentations of torsion units.

\begin{proposition}
[see \cite{Luthar-Passi} and Theorem 2.5 in \cite{Marciniak-Ritter-Sehgal-Weiss}]
\label{P:5}
Let $u\in V(\mathbb Z G)$
be of order $k$. Then $u$ is conjugate in $\mathbb
QG$ to an element $g \in G$ if and only if for
each $d$ dividing $k$ there is precisely one
conjugacy class $C$ with partial augmentation
$\varepsilon_{C}(u^d) \neq 0 $.
\end{proposition}

The next result yields that several partial augmentations are zero.

\begin{proposition}
[see \cite{Hertweck2}, Proposition 3.1; \cite{Hertweck1}, Proposition 2.2]
\label{P:4}
Let $G$ be a finite
group and let $u$ be a torsion unit in $V(\mathbb
ZG)$. If $x \in G$ and its $p$-part,
for some prime $p$, has order strictly greater
than the order of the $p$-part of $u$, then
$\varepsilon_x(u)=0$.
\end{proposition}

The main restriction on the partial augmentations is
given by the following result.

\begin{proposition}
[see \cite{Hertweck1,Luthar-Passi}]
\label{P:1}
Let either $p=0$ or $p$ is a
prime divisor of $|G|$. Suppose that $u\in V( \mathbb Z G) $ has
finite order $k$ and assume that $k$ and $p$ are coprime when
$p\neq 0$. If $z$ is a complex primitive $k$-th root of unity and $\chi$
is either a classical character or a $p$-Brauer character of $G$
then, for every integer $l$, the number
\begin{equation}\label{E:2}
\mu_l(u,\chi, p )=\textstyle\frac{1}{k} \sum_{d|k}Tr_{ \mathbb Q
(z^d)/ \mathbb Q } \{\chi(u^d)z^{-dl}\}
\end{equation}
is an integer such that \; $0 \le \mu_l(u,\chi, p ) \le \text{deg}(\chi)$.
\end{proposition}

Note that if $p=0$, we will use the notation $\mu_l(u, \chi, * )$
for $\mu_l(u,\chi, 0)$.

When $p$ and $q$ are two primes such that $G$ contains no element of order $pq$,
and $u$ is a normalized torsion unit of order $pq$, Proposition \ref{P:1} may be
reformulated for as follows. Let $\nu_k$ be
the sum of partial augmentations of $u$ with respect all conjugacy classes
of elements of order $k$ in $G$, i.e. $\nu_2 = \nu_{2a}+\nu_{2b}$, etc.
Then by (\ref{E:1}) and Proposition \ref{P:4} we obtain that
$\nu_p + \nu_q = 1$ and $\nu_k = 0$ for $k \notin \{ p, q \}$. For
each character $\chi$ of $G$ (an ordinary character or a Brauer character
in characteristic not dividing $p q$)
that is constant on all elements of orders $p$ and on all elements
of order $q$,
we have $\chi(u) = \nu_p \chi(C_p) + \nu_q \chi(C_q)$,
where $\chi(C_t)$ denote the value of the character $\chi$ on any element
of order $t$ from $G$.

From the Proposition \ref{P:1} we obtain that the values
\begin{equation}\label{E:3}
\begin{split}
\mu_l(u, \chi) =
\textstyle \frac{1}{p q} ( \; \chi(1)
& + Tr_{\mathbb Q(z^p)/ \mathbb Q} \{ \chi(u^p) z^{-pl} \} \\
& + Tr_{\mathbb Q(z^q)/ \mathbb Q} \{ \chi(u^q) z^{-ql} \}
  + Tr_{\mathbb Q(z)/ \mathbb Q} \{ \chi(u) z^{-l} \} \; )
\end{split}
\end{equation}
are nonnegative integers. It follows that if $\chi$ has the specified
property, then
\begin{equation}\label{E:4}
  \mu_l(u, \chi) =
     \textstyle \frac{1}{p q} \left( m_1 +  \nu_p m_p + \nu_q m_q \right),
\end{equation}
where
\begin{equation}\label{E:5}
\begin{split}
m_1 & = \chi(1) + \chi(C_q) Tr_{\mathbb Q(z^p)/\mathbb Q}( z^{-pl} )
              + \chi(C_p) Tr_{\mathbb Q(z^q)/\mathbb Q}( z^{-ql} ), \\
m_p & = \chi(C_p) Tr_{\mathbb Q(z)/\mathbb Q}( z^{-l} ), \qquad
m_q   = \chi(C_q) Tr_{\mathbb Q(z)/\mathbb Q}( z^{-l} ).
\end{split}
\end{equation}

Finally, we shall use the well-known bound for
orders of torsion units.

\begin{proposition}
[see \cite{Cohn-Livingstone}]
\label{P:2}
The order of a torsion element $u\in V(\mathbb ZG)$
is a divisor of the exponent of $G$.
\end{proposition}

\begin{proof}[Proof of Theorem \ref{T:1}.]
Throughout the proof we denote by $G$ the Held sporadic simple
group $\verb"He"$. It is well known \cite{Atlas,GAP} that
$|G|= 2^{10} \cdot 3^{3} \cdot 5^{2} \cdot 7^{3} \cdot 17$ and
$exp(G)= 2^{3} \cdot 3 \cdot 5 \cdot 7 \cdot 17$.
The character table of $G$, as well as the Brauer character tables
for $p \in \{2, 3, 5, 7, 17\}$
can be found by the computational algebra system GAP \cite{GAP},
which derives its data from \cite{Atlas,AtlasBrauer}.
Throughout the paper we will use the notation of GAP Character
Table Library for the characters and conjugacy classes of the
group $\verb"He"$.

It is known that $\verb"He"$
possesses elements of orders $2$, $3$, $4$, $5$, $6$, $7$, $8$,
$10$, $12$, $14$, $15$, $17$, $21$ and $28$.
From this we can derive that to deal with {\bf (KC)}, we need to show
that $V(\mathbb ZG)$ has no units of orders $34$, $35$, $51$, $85$ and $119$.
Therefore, there are also no units of orders divisible by any number from
this list.
Since by Proposition \ref{P:2}, the order
of each torsion unit divides the exponent of $G$, the only remaining
opportunities for the order are 20, 24, 30, 40, 42, 56, 60, 84, 120 and 168,
explaining formulation of part (i) of Theorem \ref{T:1}.

We will begin with orders that do not appear in $G$, and will give a
detailed proof for the order 35. The proof for the other cases can be
derived similarly from the table below, which contains the data describing
the constraints on partial augmentations $\nu_p$ and $\nu_q$
for possible orders $pq$ (including the order 35 as well)
accordingly to (\ref{E:3})--(\ref{E:5}).

%
%

If $u$ is a unit of order $35$, then $\nu_{5}+\nu_{7}=1$.
Consider 2-Brauer characters $\xi = \chi_{1}+\chi_{2}+\chi_{3}+\chi_{6}+\chi_{8}$
and $\tau = \chi_{6}+\chi_{7}+\chi_{8}+\chi_{9}$, which are encoded
in the table as $\xi=(1,2,3,6,8)_{[2]}$ and $\tau=(6,7,8,9)_{[2]}$
respectively. These characters
are constant of elements of order 5 and elements of order 7:
$\xi(C_5) = 4$, $\xi(C_{7}) = 0$, $\tau(C_5) = -8$ and $\tau(C_{7}) = 5$.
Now we obtain the system of inequalities
\[
\begin{split}
\mu_{0}(u,\xi,2) & = \textstyle \frac{1}{35}(96 \nu_{5} + 1045) \geq 0; \quad
\mu_{7}(u,\xi,2)   = \textstyle \frac{1}{35}(-24 \nu_{5} + 1025) \geq 0; \\
& \quad \mu_{0}(u,\tau,2)  = \textstyle \frac{1}{35}(-192 \nu_{5} + 120 \nu_{7} + 3090) \geq 0, \\
\end{split}
\]
which
has no nonnegative integral solution $(\nu_{5},\nu_{7})$
such that all $\mu_i(u,\chi_{j},2)$ are
non-negative integers.

The data for orders 34, 51, 85 and 119 are given in the following table.

\bigskip
\centerline{\small{
\begin{tabular}{|c|c|c|c|c|c|c|c|c|c|c|c|c|c|}
\hline
$|u|$&$p$&$q$&$\xi, \; \tau$&$\xi(C_p)$&$\xi(C_q)$&$l$&$m_1$&$m_p$&$m_q$ \\
\hline
   &   &   &                       &    &   & 0 & 5322 & 1104  & 0 \\
34 & 2 & 17& $\xi=(7, 8, 9, 12)_{[3]}$ & 69 & 0 & 2 & 5322 & -69   & 0 \\
   &   &   &                       &    &   & 17& 5184 & -1104 & 0 \\
\hline
   &   &   & $\xi=(1,2,3,6,8)_{[2]}$   & 4  & 0 & 0 & 1045 & 96   & 0 \\
35 & 5 & 7 & $\xi=(1,2,3,6,8)_{[2]}$   & 4  & 0 & 7 & 1025 & -24  & 0 \\
   &   &   & $\tau=(6,7,8,9)_{[2]}$     & -8 & 5 & 0 & 3090 & -192 & 120 \\
\hline
   &   &   & $\xi=(2,4,5)_{[*]}$       & 6  & 0 & 0 & 369  & 192  & 0 \\
51 & 3 & 17& $\xi=(2,4,5)_{[*]}$       & 6  & 0 & 17& 351  & -96  & 0 \\
   &   &   & $\tau=(24,28,33)_{[*]}$    & -7 & 0 & 1 & 5299 & -7   & 0 \\
\hline
85 & 5 & 17& $\xi=(2)_{[*]}$           & 1  & 0 & 0 & 55   & 64   & 0 \\
   &   &   &                       &    &   & 17& 50   & -16  & 0 \\
\hline
119& 7 & 17& $\xi=(1,9,15)_{[*]}$      & 2  & 0 & 0 & 7560 & 192  & 0 \\
   &   &   & $\tau=(2,3,14)_{[*]}$      & -5 & 0 & 0 & 4424 & -480 & 0 \\
\hline
\end{tabular}
}}
\bigskip

\noindent
Therefore, part (i) of Theorem \ref{T:1} is proved. Now it remains
to consider elements of orders covered by parts (ii)--(v) of Theorem \ref{T:1}.

%
%

\noindent $\bullet$ Let $|u|=5$. Since there is only one
conjugacy class in $G$ consisting of elements or order $5$,
the Proposition \ref{P:4} yields immediately that
for units of order 5 that there is precisely one
conjugacy class with non-zero partial augmentation and
by Proposition \ref{P:5} part (iii) of Theorem \ref{T:1} is proved.

%
%

\noindent$\bullet$ Let $u$ be an involution. By (\ref{E:1}) and
Proposition \ref{P:4} we get $\nu_{2a}+\nu_{2b}=1$.  Applying
Proposition \ref{P:1} to the character $\chi_{2}$ with
$\chi_2(2a)=11$, $\chi_2(2b)=3$, we obtain
\[
\begin{split}
\mu_{0}(u,\chi_{2},*) & = \textstyle \frac{1}{2} (11 \nu_{2a} + 3 \nu_{2b} + 51) \geq 0; \\ 
\mu_{1}(u,\chi_{2},*) & = \textstyle \frac{1}{2} (-11 \nu_{2a} - 3 \nu_{2b} + 51) \geq 0. \\ 
\end{split}
\]
From the requirement that all $\mu_i(u,\chi_{j},*)$ must be
non-negative integers it can be deduced that $(\nu_{2a},\nu_{2b})$
satisfies the conditions of part (ii) of Theorem \ref{T:1}.

%
%

\noindent $\bullet$ Let $u$ has order $3$. By (\ref{E:1})
and Proposition \ref{P:4} we obtain that
$\nu_{3a}+\nu_{3b}=1$. Then using Proposition \ref{P:1}
for the character $\chi_{2}$, we get the system
\[
\begin{split}
\mu_{0}(u,\chi_{2},*) & = 4 \nu_{3a} + 17 \geq 0; \qquad 
\mu_{1}(u,\chi_{2},*)   = -2 \nu_{3a} + 17 \geq 0; \\ 
& \quad \mu_{0}(u,\chi_{4},2)   = \textstyle \frac{1}{3} (-14 \nu_{3a} + 4 \nu_{3b} + 101) \geq 0, \\ 
\end{split}
\]
that has only 10 integer solutions
$(\nu_{3a},\nu_{3b})$ listed in the part (iv)
of Theorem \ref{T:1}, such that all $\mu_i(u,\chi_{j},*)$ are
non-negative integers.

%
%

\noindent$\bullet$ Let $u$ of order $17$. By (\ref{E:1})
and Proposition \ref{P:4} we get $\nu_{17a} + \nu_{17b} = 1$.
Applying (\ref{E:2}) to the ordinary character
$\chi_7$ with $\chi_7(17a)=\frac{1-\sqrt{17}}{2}$ and
$\chi_7(17b)=\frac{1+\sqrt{17}}{2}$, and to the
2-Brauer character $\chi_6$ with $\chi_6(17a)=\frac{-1+\sqrt{17}}{2}$ and
$\chi_6(17b)=\frac{-1-\sqrt{17}}{2}$, and putting
$t = 9 \nu_{17a} - 8 \nu_{17b}$, we obtain the system
of inequalities
\[
\begin{split}
\mu_{1}(u,\chi_{7},*) & = \textstyle \frac{1}{17} (- t + 1029) \geq 0; \quad 
\mu_{1}(u,\chi_{6},2)   = \textstyle \frac{1}{17} (  t + 246) \geq 0; \\ 
& \mu_{3}(u,\chi_{6},2)   = \textstyle \frac{1}{17} (-8 \nu_{17a} + 9 \nu_{17b} + 246) \geq 0, \\ 
\end{split}
\]
that has 30 integer solutions
$(\nu_{17a},\nu_{17b})$ listed in  part (v) of
Theorem \ref{T:1}, such that all $\mu_i(u,\chi_{j},*)$ are
non-negative integers. \end{proof}

\begin{proof}[Proof of Theorem \ref{T:2}.]
Throughout the proof we denote by $G$ the O'Nan sporadic simple
group $\verb"ON"$ of order
$|G|= 2^{9} \cdot 3^{4} \cdot 5 \cdot 7^{3} \cdot 11 \cdot 19 \cdot 31$
and
$exp(G)= 2^{4} \cdot 3 \cdot 5 \cdot 7 \cdot 11 \cdot 19 \cdot 31$.
Besides ordinary character tables, for $G$ also Brauer characters tables are
known for $p \in \{2, 3, 5, 7, 11, 19, 31\}$ (see \cite{Atlas,GAP,AtlasBrauer}).
As before, we use the GAP notation for the characters and conjugacy classes of $G$.

It is known that $G$ possesses elements of orders
1, 2, 3, 4, 5, 6, 7, 8, 10, 11, 12, 14, 15, 16, 19, 20, 28 and 31.
The Kimmerle conjecture
requires us to consider possible units of orders
$21$, $22$, $33$, $35$, $38$, $55$, $57$, $62$, $77$, $93$,
$95$, $133$, $155$, $209$, $217$, $341$ and $589$. As in the
proof of Theorem \ref{T:1}, below we give the table containing
the data describing the constraints on partial augmentations
$\nu_p$ and $\nu_q$ accordingly to (\ref{E:3})--(\ref{E:5}) for
all these orders, except order 22. From this table
parts (i), (v) and (vi) of Theorem \ref{T:2}
are derived in the same way as in the proof of Theorem \ref{T:1},
except orders 22 and 35 which will be treated separately.
Since we are not able to prove the non-existence of units of orders
33 and 57, and also do not consider orders with more than two
prime factors, the condition $\exp(G)\equiv 0 \pmod{|u|}$ (see
Proposition \ref{P:2}) results in listing orders
24, 30, 33, 40, 48, 56, 57, 60, 80, 112, 120 and 240
in the ``exclusive'' part (i) of Theorem \ref{T:2}.

\bigskip
\centerline{\small{
\begin{tabular}{|c|c|c|c|c|c|c|c|c|c|c|c|c|c|}
\hline
$|u|$&$p$&$q$&$\xi, \; \tau$&$\xi(C_p)$&$\xi(C_q)$&$l$&$m_1$&$m_p$&$m_q$ \\
\hline
   &   &   &                       &    &   & 0 & 98493 & 312  & 0 \\
21 & 3 & 7 & $(1, 3, 9, 10)_{[5]}$ & 26 & 0 & 1 & 98415 & 26   & 0 \\
   &   &   &                       &    &   & 7 & 98415 & -156 & 0 \\
\hline
   &   &   & $\xi=(3)_{[7]}$       & 6  & 0 & 0 & 1233  & 120  & 0 \\
33 & 3 & 11& $\xi=(3)_{[7]}$       & 6  & 0 & 11& 1215  & -60  & 0 \\
   &   &   & $\tau=(23)_{[7]}$     & 4  & 0 & 0 & 143382& 80   & 0 \\
\hline
35 & 5 & 7 & $\xi=(1,3)_{[3]}$     & -2 & 0 & 0 & 335   & -48  & 0 \\
   &   &   &                       &    &   & 7 & 345   & 12   & 0 \\
\hline
   &   &   &                       &    &   & 0 & 70358 & 4806  & 0 \\
38 & 2 & 19& $\xi=(2,7,8)_{[*]}$   & 267& 0 & 1 & 69824 & 267   & 0 \\
   &   &   &                       &    &   & 19& 69824 & -4806 & 0 \\
\hline
   &   &   &                       &    &   & 0 & 415 & 80  & 0 \\
55 & 5 & 11& $\xi=(1,2)_{[7]}$     & 2  & 0 & 5 & 415 & -8  & 0 \\
   &   &   &                       &    &   & 11& 405 & -20 & 0 \\
\hline
   &   &   & $\xi=(1,2,8,9)_{[7]}$     & 18 & 0 & 0 & 21924 & 648  & 0 \\
57 & 3 & 19& $\xi=(1,2,8,9)_{[7]}$     & 18 & 0 & 0 & 36369 & -972 & 0 \\
   &   &   & $\tau=(23)_{[7]}$         & 4  & 0 & 1 & 143370& 4    & 0 \\
\hline
   &   &   &                       &    &   & 0 & 150 & -150& 0 \\
62 & 2 & 31& $\xi=(1,2)_{[3]}$     & -5 & 0 & 2 & 150 & 5   & 0 \\
   &   &   &                       &    &   & 31& 160 & 150 & 0 \\
\hline
77 & 7 & 11& $\xi=(1,3)_{[3]}$     & 0  & 2 & 0 & 363 & 0   & 120 \\
   &   &   &                       &    &   & 11& 363 & 0   & -20 \\
\hline
   &   &   &                       &    &   & 0 & 26799 & -1380 & 0 \\
93 & 3 & 31& $\xi=(1,4,5)_{[2]}$   & 23 & 0 & 3 & 26799 & -46   & 0 \\
   &   &   &                       &    &   & 31& 26730 & -690  & 0 \\
\hline
\end{tabular}
}}

\bigskip
\centerline{\small{
\begin{tabular}{|c|c|c|c|c|c|c|c|c|c|c|c|c|c|}
\hline
$|u|$&$p$&$q$&$\xi, \; \tau$&$\xi(C_p)$&$\xi(C_q)$&$l$&$m_1$&$m_p$&$m_q$ \\
\hline
95 & 5 & 19& $\xi=(1,2)_{[3]}$         & 0  & 3 & 0 & 209 & 0   & 216 \\
   &   &   &                       &    &   & 19& 209 & 0   & -54 \\
\hline
133& 7 & 19& $\xi=(1,3)_{[3]}$         & 0  & 1 & 0 & 361 & 0   & 108 \\
   &   &   &                       &    &   & 19& 361 & 0   & -18 \\
\hline
   &   &   &                       &    &   & 0 & 480 & -480 & 0 \\
155& 5 & 31& $\xi=(2,3)_{[3]}$         & -4 & 0 & 5 & 480 & 16   & 0 \\
   &   &   &                       &    &   & 31& 500 & 120  & 0 \\
\hline
209& 11& 19& $\xi=(2,5)_{[3]}$         & 0  & 3 & 0 & 703 & 0   & 540 \\
   &   &   &                       &    &   & 19& 703 & 0   & -54 \\
\hline
217& 7 & 31& $\xi=(1,3)_{[3]}$         & 0  & 2 & 0 & 403 & 0   & 360 \\
   &   &   &                       &    &   & 31& 403 & 0   & -60 \\
\hline
341& 11& 31& $\xi=(1,2)_{[3]}$         & 1  & 0 & 0  & 165 & 300 & 0 \\
   &   &   &                       &    &   & 31 & 154 & -30 & 0 \\
\hline
589& 19& 31& $\xi=(1,2)_{[3]}$         & 3  & 0 & 0  & 209 & 1620 & 0 \\
   &   &   &                       &    &   & 31 & 152 & -90  & 0 \\
\hline
\end{tabular}
}}
\bigskip

\noindent
$\bullet$ Let $u$ be a torsion unit of order $22$. By (\ref{E:1})
and Proposition \ref{P:4} we have that $\nu_{2a}+\nu_{11a}=1$.
Then using Proposition \ref{P:1} for the ordinary character
$\chi_2$ and 3-Brauer character $\chi_2$ and putting
$t=64 \nu_{2a} -  \nu_{11a}$, we obtain the system
\[
\begin{split}
\mu_{0}(u,\chi_{2},3) & = \textstyle \frac{1}{22} (-60 \nu_{2a} + 148) \geq 0; \qquad 
\mu_{11}(u,\chi_{2},3)  = \textstyle \frac{1}{22} (60 \nu_{2a} + 160) \geq 0; \\ 
\mu_{1}(u,\chi_{2},*) & = \textstyle \frac{1}{22} ( t + 10881) \geq 0; \qquad \qquad 
\mu_{11}(u,\chi_{2},*)  = \textstyle \frac{5}{11} (-t + 1087) \geq 0; \\ 
& \qquad \qquad \mu_{0}(u,\chi_{2},*)   = \textstyle \frac{1}{22} ( 10 t + 10998) \geq 0, \\ 
\end{split}
\]
which has no integral solutions such that all $\mu_i(u,\chi_{j},*)$ are
non-negative integers.

\noindent $\bullet$ Let $|u|=35$.
By (\ref{E:1}) and Proposition \ref{P:4}, $\nu_{5a}+\nu_{7a}+\nu_{7b}=1$.
Also, from the table above we have that $\nu_{5a}=-20$. This restriction
greatly facilitates the next step, when using Proposition \ref{P:1} for
the ordinary character $\chi_2$ with $\chi_2(7a)=17$, $\chi_2(7b)=3$
and putting $t = 17 \nu_{7a} + 3 \nu_{7b}$, we obtain two incompatible
constraints
\[
\begin{split}
\mu_{0}(u,\chi_{2},*)  = \textstyle \frac{1}{35} (24 t + \alpha_1) \geq 0; \qquad 
\mu_{7}(u,\chi_{2},*)  = \textstyle \frac{1}{35} (-6 t + \alpha_2) \geq 0, \\ 
\end{split}
\]
where the values of $\alpha_1$ and $\alpha_2$
parametrised by $\chi(u^{5})=k_1 \chi(7a) + k_2 \chi(7b)$
are given in the following table.

\bigskip
\centerline{\small{
\begin{tabular}{|c|c|c|c|}
\hline
$\quad (k_1,k_2) \quad $ & $\qquad(\alpha_1, \alpha_2)\qquad$ &
$\quad (k_1,k_2) \quad $ & $\qquad(\alpha_1, \alpha_2)\qquad$ \\ \hline
(1,0)    & (11522,10927) & (11,-10) & (12362,11767) \\ \hline
(0,1)    & (11438,10843) &(10,-9) &(12278,11683)\\ \hline
(22,-21) & (13286, 12691)& (9,-8)& (12194,11599)\\ \hline
(21,-20) & (13202,12607)& (8,-7) & (12110,11515) \\ \hline
(20,-19) & (13118,12523)& (7,-6)& (12026,11431)\\ \hline
(19,-18)& (13034,12439)& (6,-5)& (11942,11347)\\ \hline
(18,-17) & (12950,12355)&(5,-4) &(11858,11263)\\ \hline
(17,-16) & (12866,12271)& (4,-3)& (11774,11179)\\ \hline
(16,-15) & (12782,12187) & (3,-2)& (11690,11095)\\ \hline
(15,-14) &(12698,12103)& (2,-1)  & (11606,11011) \\ \hline
(14,-13)& (12614,12019)  & (-1,2) & (11354,10759)\\ \hline
(13,-12) & (12530,11935) &(-2,3)&(11270,10675)\\ \hline
(12,-11) & (12446, 11851)& (-3,4)& (11186, 10591)\\ \hline
\end{tabular}
}}
\bigskip

\newpage
Now it remains to consider elements of orders that appear in the group $G$.

\noindent $\bullet$ Let $|u| \in \{ 2,3,5,11 \}$.
Since there is only one conjugacy class in $G$
consisting of elements or each of these orders,
the Proposition \ref{P:4} yields immediately that
for such units there is precisely one
conjugacy class with non-zero partial augmentation and
by Proposition \ref{P:5} part (ii) of Theorem \ref{T:2} is proved.

\noindent $\bullet$ Let $u$ has order $7$. By (\ref{E:1})
and Proposition \ref{P:4} we obtain that
$\nu_{7a}+\nu_{7b}=1$. Put $t=-17 \nu_{7a} - 3 \nu_{7b}$.
Then using Proposition \ref{P:1}
for the ordinary character $\chi_2$ with values
$\chi_2(7a)=17$, $\chi_2(7b)=3$ and 3-Brauer character $\chi_2$
with $\chi_2(7a)=7$, $\chi_2(7b)=0$
we get the system of inequalities
\[
\begin{split}
\mu_{0}(u,\chi_{2},*) & = \textstyle \frac{1}{7} ( -6t + 10944) \geq 0; \quad \quad  
\mu_{1}(u,\chi_{2},*)   = \textstyle \frac{1}{7} ( t + 10944) \geq 0; \\ 
\mu_{0}(u,\chi_{2},3) & = \textstyle \frac{1}{7} (42 \nu_{7a} + 154) \geq 0; \qquad 
\mu_{1}(u,\chi_{2},3)   = \textstyle \frac{1}{7} (-7 \nu_{7a} + 154) \geq 0, \\ 
\end{split}
\]
that has 26 integer solutions
$(\nu_{7a},\nu_{7b})$ listed in the part (iii)
of Theorem \ref{T:2}, such that all $\mu_i(u,\chi_{j},*)$ are
non-negative integers.

\noindent $\bullet$ Let $|u|=31$. By (\ref{E:1})
and Proposition \ref{P:4} we obtain that
$\nu_{31b}=1-\nu_{31a}$.
Proposition \ref{P:1} for the 7-Brauer character $\chi_3$ with \;
$\chi_3(31a) = \overline{\chi_3(31b)} = \frac{-7+\sqrt{-31}}{2}$
\; yields
\[
\begin{split}
\mu_{1}(u,\chi_{3},7)  = \textstyle \frac{1}{31} (31 \nu_{31a} + 1209) \geq 0; \quad 
\mu_{3}(u,\chi_{3},7)  = \textstyle \frac{1}{31} (-31 \nu_{31a} + 1240) \geq 0. \\ 
\end{split}
\]
This system of inequalities has 80 integer solutions
$(\nu_{31a},\nu_{31b})$ listed in the part (iv)
of Theorem \ref{T:2}, such that all $\mu_i(u,\chi_{j},*)$ are
non-negative integers.
\end{proof}

\bibliographystyle{plain}
\bibliography{HeON_Kimmerle_JAMS}

\end{document}